\documentclass[11pt]{article}
\usepackage[margin=2.5cm]{geometry}
\usepackage{amsmath,amsthm,amsfonts,amssymb}
\usepackage[mathscr]{euscript}
\usepackage{amsmath}
\usepackage{ragged2e}
\usepackage{array,float}

\usepackage{graphicx,latexsym}
\usepackage{lscape}
\usepackage{multicol}
\usepackage{multirow}
\usepackage{graphicx}
\usepackage{subfig}
\usepackage{cite}

\newtheorem{Theorem}{Theorem}

\newtheorem{Corollary}{Corollary}

\newtheorem{discu}{Discussion:}

\newtheorem{conje}{Conjecture:}

\begin{document}
\date{}
\centering
\title{\textbf{Certain Domination Parameters and its Resolving Version of Fractal Cubic Networks}}
\author{\begin{tabular}{rlc}
		\textbf{S. Prabhu$^{\text a, }$\thanks{Corresponding author: drsavariprabhu@gmail.com}, A.K. Arulmozhi$^{\text b}$, M. Arulperumjothi$^{\text  c}$ } \\
	\end{tabular}\\
	\begin{tabular}{c}
		\small $^{\text a}$Department of Mathematics, Rajalakshmi Engineering College, Thandalam, Chennai 602105, India \\
		$^{\text b}$\small Department of Mathematics, R.M.K. College of Engineering and Technology, Puduvoyal 601206, India\\
		$^{\text c}$\small Department of Mathematics, St. Joseph's College of Engineering, Chennai 600119, India \\
\end{tabular}}
\maketitle
\vspace{-0.5 cm}
\begin{abstract}
Networks are designed to communicate, operate and allocate the tasks to the respective commodities. Operating the supercomputers became challenging, and it was handled by the network design commonly known as hypercube, denoted by $Q^n$. In a recent study, the hypercube networks were not enough to hold the parallel processors in the supercomputers. Thus, variants of hypercubes were discovered to produce an alternative to the hypercube. A new variant of the hypercube, the \textit{fractal cubic network}, can be used as the best alternative in the case of hypercubes, which was wrongly defined in [Eng. Sci. Technol. \textbf{18}(1) (2015) 32--41]. Arulperumjothi et al. recently corrected this definition and redefined the network in [Appl. Math. Comput. \textbf{452} (2023) 128037]. Our research investigates that the fractal cubic network is a \textit{rooted product} of two graphs. We try to determine its domination and resolving domination parameters, which could be applied to resource location and broadcasting-related problems.
\hspace*{0.5cm}
\end{abstract}
\justifying

\hspace{0.5cm}\textbf{Keywords:} {\small resolving sets; dominating sets; resolving domination; fractal cubic network}\\ 
{\bf AMS Subject Classification:} 05C12 $\cdot$ 05C69

\section{Introduction and Motivation}

Graph frameworks, consisting of a network of connections, are extensively employed in various dynamic, circuit-related, genetic, and chemical systems. It aids in modeling the transmitters of brain systems. The structure of a graph consists of vertices and edges. Each vertex describes a node in the network, and each edge represents a link between nodes. The interconnection network is a sophisticated linkage between the array of processors and the communication pathways connecting any two distinct processors. The network is employed to exchange data across processors in parallel network computation. Network dependability is the most critical component in designing the network's geometry. The interconnection network is an essential subsystem for high-performance computing systems and data centers \cite{AkKr89}.
Consequently, contemporary suggestions for the interconnection network must ensure minimal latency overhead and maximal transmission bandwidth. Occasionally, it is impractical to execute and evaluate new designs physically; they must be examined and validated using data-driven software tools, such as network simulators for connectivity information graphs. Interconnection networks include numerous primary and intrinsic uses in system-designed architectures, although they are predominantly employed in parallel computing architecture.

Interconnection networks with multiprocessors are often crucial for connecting many reliably replicated processors. Message passing is predominantly employed in place of shared memory to provide comprehensive transmission and synchronization across processors for planned execution.  The graph $\Gamma$ can be shown such that every pair of vertices is directly linked via transmission links. The metrics employed to assess the efficacy of the structure include bisection width, broadcasting duration, degree, diameter, and fault tolerance \cite{AkKr89}. Let $\Gamma$ be a connected graph with $V(\Gamma)$ (vertex set) and $E(\Gamma)$ (edge set), respectively. For any $s\in V(\Gamma)$, $N_{\Gamma}(s)=\{t\in V(\Gamma) \enspace | \enspace st\in E(\Gamma)\}$ and $N_{\Gamma}[s] = N_{\Gamma}(s)\cup\{s\}$ as an \textit{open} and \textit{closed neighborhoods} of $s$ respectively. If $\Gamma$ is understood we denote $N_{\Gamma}(s) (N_{\Gamma}[s])$ as $N(s)(N[s])$ respectively. The \textit{degree} of $s$ is defined by $d_{\Gamma}(s) =|N_{\Gamma}(s)|$. $\Delta$ symbolizes the greatest degree. For $k$-regular graph, $ \Delta = k=\delta$, where $\delta$ is the least degree. For a subset $D$ of $V(\Gamma)$, $\Gamma[D]$ is the \textit{induced sub graph} of $\Gamma$. For a graph $\Gamma$ with $n$ vertices, the \textit{order} is computed as $n$, which could be referred to as $|V(\Gamma)|=n$. Denote $\mathbb{N}_n:=\{1,2,\ldots, n\}$. Two vertices $e$ and $f$ are said to be \textit{false twins} if $N(e)=N(f)$ and \textit{true twins} if $N[e]=N[f]$. Two vertices in the connected graph $\Gamma$ are said to be \textit{twins} if they are either \textit{true or false twins}. A set $T \subseteq V(\Gamma)$ is said to be a \textit{open (closed)} twin set if every pair of vertices in $T$ are \textit{false (true)} twins in $\Gamma$. 

\section{Resolving and Dominating Sets}

Resolving sets provide a mechanism for identifying the origin of diffusion in a network. Identifying the source of a disease disseminated through a community could be exceedingly beneficial in numerous contexts. Although the resolving set  (\textbf{RS}) provide a solution when inter-node intervals and starting spread time are established, resolvability must be broadened to include random start timings and random nodal communication delays.

For $R=\{m_{1},m_{2}, \ldots ,m_{t}\}\subseteq V(\Gamma)$, the \textit{code} of $j\in V(\Gamma)$ with respect to $R$ is defined as the $t$-vector

\begin{equation*}
	C_{R}(j)=\Big(d(j,m_{1}),d(j,m_{2}), \ldots ,d(j,m_{t})\Big),
\end{equation*}

where $d(j,k)$ denote the distance between $j$ and $k$. The collection $R$ is termed a \textbf{RS} for $\Gamma$ if, for every pair of different vertices $g,h \in V(\Gamma)$, the codes $C_{R}(g)$ and $C_{R}(h)$ are unique. A set $R$ is a \textbf{RS} for $\Gamma$ if, for any pair of vertices $g$ and $h$ in $V(\Gamma)$, $\exists$ $r \in R$ such that the distances $d(g,r)$ and $d(h,r)$ are not equal. Consult Figure~\ref{RS}. Among all potential \textbf{RS} for $\Gamma$, the ones with the smallest size are of particular interest, referred to as a \textit{basis}. The size of the smallest \textbf{RS} is referred to as the \textit{metric dimension} of $\Gamma$, indicated by $\dim(\Gamma)$. The problem of finding resolving sets remains NP-complete for general graphs \cite{HaMe76}.

\begin{figure}[H]
	\centering
	\includegraphics[scale=0.6]{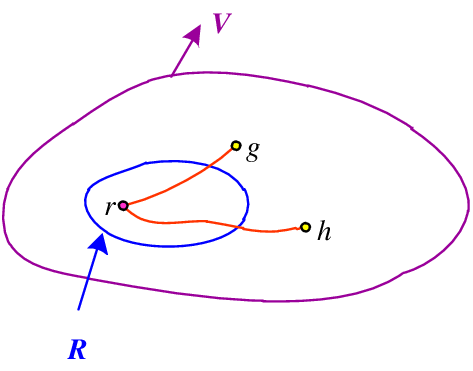}  
	\caption{Resolving set $R$} \label{RS}
\end{figure}

A \textit{dominating subset} $D$ of $V(\Gamma)$, in which every member of $\Gamma$ is in $D$ or adjacent to at least one vertex of $D$. The least size of $D$ is known as the \textit{domination number} and is represented as $\gamma$ \cite {HaHeSl13}. By imposing constraints on a $D$, several domination parameters are laid. Some of them are independent, total, connected, double and 2-domination. A dominating set \textbf{(DS)} is \textit{independent dominating set} \textbf{(IDS)} if $\Gamma[D]$ is a null graph. We say \textbf{DS} as a \textit{total dominating set} \textbf{(TDS)} if each member in $V$ is connected to some member belonging to $D$ and \textit{connected dominating set} \textbf{(CDS)} if $\Gamma[D]$ is connected. If $|N[v] \cap D| \geq 2$ for any $v \in V(\Gamma)$, then $D$ is a \textit{double dominating set} \textbf{(DDS)} of $\Gamma$ and if $|N(v) \cap D| \geq 2$ for $v \in V(\Gamma) \setminus D$, then $D$ is a 2-\textit{dominating set} \textbf{(2DS)} of $\Gamma$. The least size of these sets were respectively denoted as $\gamma_i$, $\gamma_{t}$, $\gamma_{c}$, $\gamma_{\times 2}$ and $\gamma_{2}$. Let $U,V \subseteq V(\Gamma)$, where $(U,V)$ is an ordered pair of disjoint sets $U$ and $V$ is a \textit{quasi-double dominating pair} of $\Gamma$ if $U \cup V \in D_2$-set and $V \in D_{\times2}$-set of $\Gamma - U$. Then \textit{quasi-double domination number} denoted by $\gamma_{q\times2}$ is $\min\{|U|+|V|: U\cup V \in D_2$-set and $V \in D_{\times2}$-set of $(\Gamma - U)\}$. Domination remains NP-complete and has application in numerous areas like communication systems\cite{BrMaYa19}, resource location problems\cite{BaOrRe15, FaAsHe12}, social networks\cite{KeCo88}, and models of biological networks \cite{MiMeBo11, NaAk16, GaWuSi18}. Determining the domination number of a $r$-dimensional hypercube ($Q^r$) is fundamental in coding, graph theory and circuit-related sciences. Domination invariants can demonstrate types of scattering-related problems on multi-level interconnection networks, where the hypercube in opportunity serves as an essential paradigm for graph networks.

In graph theory literature, various parameters on resolving sets and dominating sets were identified and extensively studied. They are applied to parallel computing architectures and neural networks for solving resource location, image processing, and chemical-related problems. Specific variants of resolving sets are path resolving set, connected resolving set, independent resolving set, one-factor resolving set and one-size resolving set. In the technological era, combining the concepts and logic creates new variations for solving higher-level problems, and one such idea is the \textit{resolving domination} \textbf{(RD)} in which the set $D \subseteq V(\Gamma)$ is both resolving and dominating set. The \textit{resolving domination number} denoted by $\gamma_r$ is the least size of resolving dominating set \cite{BrChDu03, HeMoPe19}. Specific variants of resolving domination are discussed, including resolving independent domination, resolving total domination and resolving connected domination. The least cardinality of these sets was represented as $\gamma_{ri}, \gamma_{rt}$ and $\gamma_{rc}$ respectively. We abbreviate resolving dominating set as \textbf{RDS}, resolving independent dominating set as \textbf{RIDS}, resolving total dominating set as \textbf{RTDS} and resolving connected dominating set as \textbf{RCDS} respectively. Since domination and resolving set problems remain NP-complete, resolving domination problems are also a class of NP-complete problems for general graphs. Brigham et al. in 2003, introduced the concept of \textbf{RD} and provided the lower and upper bounds, \textbf{RD} number of standard graphs, relationship with the diameter, order, clique number and characterized \textbf{RD} number for $n-1$\cite{BrChDu03}. Monsanto and Rara established \textbf{RD} number of certain graphs under some binary operations \cite{MoRa23}. In 2015, resolving connected domination was introduced by Naji and Soner, and proved primary results on resolving connected domination \cite{NaSo15}. For further resolving and domination-related problems, the readers could refer to \cite{PrFlAr18, PrJeAr23, PrMaDa24, PrMaAr22, PrJa24, PrJaKl24, PrDeEl22, PrDeAr22, PrArHe24, HeMoPe19}.

\section{Fractal Cube: A Fascinating Variant of Hypercube}

The examination of self-similarity and fractality in discrete systems, especially complex networks, has intensified. This increase in interest is driven by theoretical advancements in complex network theory and the practical requirements of real-world applications. Translating the ideas of fractal geometry from general topology, which addresses continuous or infinite objects, to finite structures in a mathematically rigorous manner presents a significant difficulty. The investigation of fractals enhances our comprehension of the intrinsic beauty and intricacy of the natural world while also possessing extensive applicability across multiple scientific fields, including biology, physical sciences, computer networks,  and chemical graph theory. Sierpiński-type structures have been thoroughly investigated in fractal theory and application, with substantial research illustrating its significance and utility across various domains \cite{EsRo19, PrJaKl24}. For more on recent work of fractal networks can be found in \cite{ZeHuGu23, Al20, ImGaFa17, ImJa20, GhNaIs23, AlIm24}. 

\begin{figure}[H]
	\centering
	\includegraphics[scale=0.6]{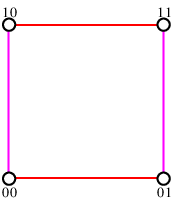}
	\quad \quad  \quad
	\includegraphics[scale=0.6]{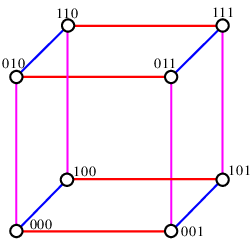}
	\quad \quad \quad
	\includegraphics[scale=0.6]{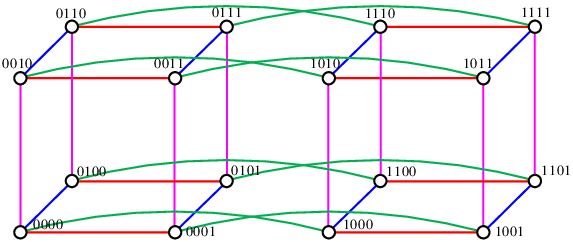}
	\caption{Various dimensions of hypercubes}
	\label{Q}
\end{figure}

The cube ($Q^n$) is a prevalent architecture characterised by its regularity, transit efficiency, recursive  configuration, symmetry, and elevated connectivity. See Figure \ref{Q} for the hypercube of dimensions 2, 3 and 4. In recent years, hypercubes have been extensively studied for their diverse features \cite{HaHaWu88}.

\begin{table}[H]
	\centering
	\caption{Literature on variants of hypercubes} \label{hyp}
	\scalebox{0.80}{
		\begin{tabular}{|c|l|c|}
			\hline
			S.No                     & \multicolumn{1}{c|}{Variants of hypercubes} & References              \\ \hline
			1                        & Exchanged hypercubes                        & \cite{LoHsPa05}         \\ \hline
			2                        & Folded hypercubes                           & \cite{ElLa91, ZhLiXu08} \\ \hline
			3                        & Crossed cubes                               & \cite{Ef92}             \\ \hline
			4                        & Exchanged crossed cubes                     & \cite{FaJi07, LiMuLi13} \\ \hline
			5                        & Twisted cubes                               & \cite{AbPa91, ChWaHs99} \\ \hline
			6                        & M\"{o}bius cubes                       & \cite{CuLa95}         \\ \hline
			7                        & Spined cubes                               & \cite{ZhFaJi11}         \\ \hline
			8                        & Locally twisted cubes                                & \cite{YaEvMe05}         \\ \hline
			9                        & Shuffle cubes                                & \cite{LiTaHs01}           \\ \hline
			10                       & Augmented cubes                             & \cite{ChSu02}           \\ \hline
			11                       & Hierarchical cubic networks                 & \cite{GhDe95, YuPa98}   \\ \hline
			12 & Folded hierarchical cubic networks          & \cite{DuChFa95}         \\ \hline
	\end{tabular}}
\end{table}

In Intel's hypercube, the new node functions as the cube manager, with direct links to all processors within the system, analogous to a physical machine. In this case, the bisection width of this architecture is $2^{n-1}$. In parallel architecture, hypercube fails to have some property. For example it has high bisection width and non-constant node degree. The literature presents numerous variations as shown in Table \ref{hyp}. 

Although the hypercube versions listed above have been the subject of several research studies, none of them, except the fractal cubic network, have examined their problem and its resolving number. Several variations on hypercube have been proposed in the literature, but unfortunately, none of them has less bisection width and constant node degree. Fractal cubic network  \textbf{(FCN)} is an ideal architecture, for its constant node degree and low bisection width.  While the concept of this structure is ambiguous in~\cite{KaSe15}, it was rectified in~\cite{ArKlPr23}. Motivated by this, we recently examined power domination and resolving power domination \cite{PrArHe24} for this recently introduced hypercube variation fractal cubic network. The authors of \cite{ArKlPr23},  characterise $FCN(0)$ as a cycle of vertex set with cardinality four: $00$, $01$, $11$, and $10$. For \( l \ge 1 \), and define \( FCN(l) \) as follows:

An $l$-dimensional \textbf{FCN} is defined as $FCN(l) = (V_{1}(l), E_{1}(l))$, $l > 0$, and can be designed as follows
\[
FCN(l) = 11 \mathbin\Vert FCN(l-1)\cup 01 \mathbin\Vert FCN(l-1)\cup 10 \mathbin\Vert FCN(l-1) \cup 00 \mathbin\Vert FCN(l-1),
\]
where
\[
V_{1}(l) = 11 \mathbin\Vert V_{1}(l-1)\cup 01 \mathbin\Vert V_{1}(l-1)\cup 10\mathbin\Vert V_{1}(l-1)\cup 00\mathbin\Vert V_{1}(l-1)
\]
and
\[
\begin{array}{lcl}
	E_{1}(l) & = & 11 \mathbin\Vert E_{1}(l-1) \cup 01 \mathbin\Vert E_{1}(l-1)\cup 10\mathbin\Vert E_{1}(l-1)\cup 00\mathbin\Vert E_{1}(l-1) \\
	& & \, \cup \, \{(00100101\ldots01,10100101\ldots01), (10100101\ldots01,11100101\ldots01) \} \\
	& & \, \cup \, \{(11100101\ldots01,01100101\ldots01),(01100101\ldots01,00100101\ldots01)\}.
\end{array}
\]
Figure~\ref{fcn} respectively, denotes $FCN$ dimensions of 0,1 and 2 and $\mathbin\Vert$ denotes the concatenation operator and each string is of length $2l+2$ for dimension of $l$.

Operation on graphs and design of new networks, which could be used as an alternative to multistage interconnection networks. Several operations on graphs deal with merging the vertex and adding additional edges. A few operations on graphs are listed here: Cartesian product, strong product, corona product, lexicographic product, tensor product and rooted product. Here we concentrate on the rooted product operation and is defined as follows. A graph is usually called rooted if one of its nodes is designated as a root to set it apart from the other nodes. Let $\Omega_i$, $i \in \mathbb{N}_n$ be the $n$ copies of $\Omega$ and let $\Gamma$ be an $n$ ordered graph. The graph $\Gamma \circ_v \Omega$ is generated by assigning a unique vertex $v$ to each $\Omega_i$ on the $i^{\text{th}}$ node of $\Gamma$ and the graph obtained is the rooted product of $\Gamma$ and $\Omega$ respectively. It is interesting to note that the hypercube is operated by the Cartesian product of $K_2$ and its lower dimensional hypercube. That is $Q^n = K_2 \times Q^{n-1}$. The new variant of hypercube \textbf{FCN}, is constructed by a rooted product of $C_4$ and its lower dimension of \textbf{FCN}.  Thus \textbf{FCN} can be represented as $FCN(l) = C_4 \circ_v FCN(l-1), v \in \{0010(01)^{l-1}, 0110(01)^{l-1}, 1010(01)^{l-1}, 1110(01)^{l-1} \}$. Godsil and McKay \cite{GoMc78} coined rooted product operation of two graphs, and further, the readers could refer to \cite{KuLeYe16, CaRo20, CaEs23, CaRuSe24} for detailed literature on domination variants in the rooted product of graphs. 

\begin{figure}[H]
	\centering
	\includegraphics[scale=0.55]{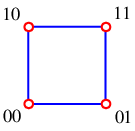}
	\quad \quad  \quad
	\includegraphics[scale=0.55]{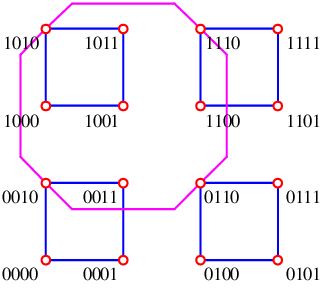}
	\quad \quad \quad
	\includegraphics[scale=0.55]{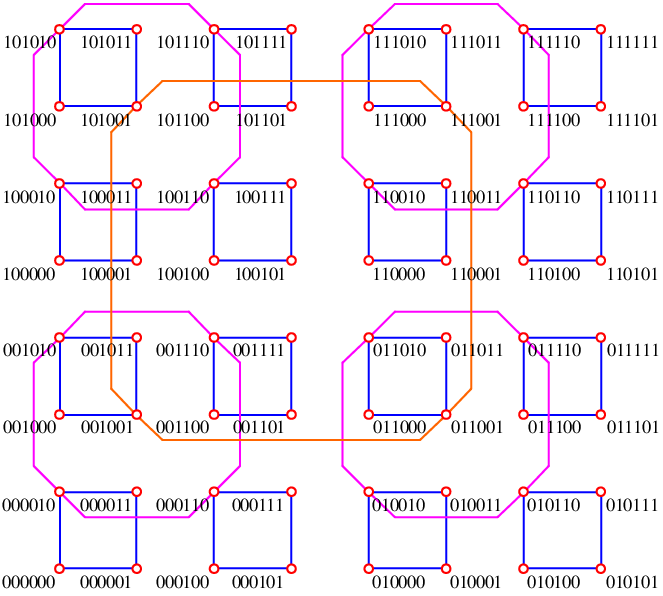}
	\caption{Various dimensions of $FCN$}
	\label{fcn}
\end{figure}

\section{Main Results}

\begin{Theorem}{\rm \cite{PrFlAr18}} \label{22}
	Let $\Gamma$ be a connected graph with twin sets $T_k$, $1 \leq k \leq p$, then $\dim(\Gamma) \geq \sum_{i=1}^{p}|T_k|-p$. 
\end{Theorem}

\begin{Theorem}{\rm \cite{ArKlPr23}}\label{mmdfcn}
	For $l>0$,  $\dim(FCN(l))=4^{l}$.
\end{Theorem}

\begin{Theorem} {\rm \cite{BrChDu03}} \label{1}
	For an isolate-free graph $\Gamma$, $\max\{ \gamma(\Gamma), \dim(\Gamma) \} \leq \gamma_r(\Gamma) \leq \gamma(\Gamma)+ \dim(\Gamma)$.
\end{Theorem}

\begin{Theorem} {\rm \cite{KuLeYe16}} \label{2}
	Let $\Gamma$ be an isolate-free graph and $n \geq 2$. Then for any graph $\Omega$ with root $v$ and $|V(\Omega)|\geq 2$, $\gamma(\Gamma \circ_v \Omega) \in \{n\gamma(\Omega), n\gamma(\Omega)-n+ \gamma(\Gamma) \}$.
\end{Theorem}

\begin{Theorem} {\rm \cite{KuLeYe16}} \label{20}
	Let $\Gamma$ be an isolate-free graph and $n \geq 2$. Then for any graph $\Omega$ with root $v$ and $|V(\Omega)|\geq 2$, $\gamma_i(\Gamma \circ_v \Omega) \in \{n\gamma_i(\Omega), n\gamma_i(\Omega)-n+ \gamma_i(\Gamma) \}$.
\end{Theorem}

\begin{Theorem} {\rm \cite{CaRo20}}  \label{3}
	Let $\Gamma$ and $\Omega$ be two isolate-free graphs. For any $v \in V(\Omega)$, $\gamma_t(\Gamma \circ_v \Omega) \in \{n\gamma_t(\Omega)-n, \gamma(\Gamma)+n\gamma_t(\Omega)-n, \gamma_t(\Gamma)+n\gamma_t(\Omega)-n, n\gamma_t(\Omega) \}$.
\end{Theorem}
\begin{Theorem} {\rm \cite{KuLeYe16}} \label{4}
	Let $\Gamma$ be an isolate-free graph and $n \geq 2$. Then for any graph $\Omega$ with root $v$ and $|V(\Omega)|\geq 2$, $\gamma_c(\Gamma \circ_v \Omega) \in \{n\gamma_c(\Omega), n\gamma_c(\Omega)+n \}$.
\end{Theorem}
\begin{Theorem} {\rm \cite{CaEs23}} \label{5}
	Let $\Gamma$ and $\Omega$ be two connected graphs. If $v \in V(\Omega)$, then $\gamma_{\times 2}(\Gamma \circ_v \Omega) \in \{n\gamma_{\times 2}(\Omega), \gamma_{q \times 2}(\Gamma) + n\gamma_{\times 2}(\Omega)-n, \gamma_{2}(\Gamma) + n\gamma_{\times 2}(\Omega)-n, \gamma(\Gamma) + n\gamma_{\times 2}(\Omega)-n, n\gamma_{\times 2}(\Omega)-n, \gamma_{\times 2}(\Gamma) + n\gamma_{\times 2}(\Omega)-2n \}$.
\end{Theorem}

\begin{Theorem} {\rm \cite{CaEs23}} \label{6}
	Let $\Gamma$ be an isolate-free graph, then $\gamma_{\times 2}(\Gamma \circ_v C_4) = n\gamma_{\times 2}(C_4)$.
\end{Theorem}

\begin{Theorem} {\rm \cite{CaRuSe24}} \label{7}
	Let $\Gamma$ and $\Omega$ be any graph. For any vertex $v \in V(\Omega)$, $\gamma_2(\Gamma \circ_v \Omega) \in \{ \gamma(\Gamma)+ n\gamma_2(\Omega)-n, \gamma_2(\Gamma)+n\gamma_2(\Omega)-n, n\gamma_2(\Omega) \}$.
\end{Theorem}

\begin{Theorem} {\rm \cite{CaRuSe24}} \label{8}
	Let $\Gamma$ be an isolate-free graph such that $\gamma_2(\Gamma) < n$ and let $\Omega$ be a graph and $v \in V(\Omega)$. Then {\rm(a)} and {\rm(b)} are equivalent:
	\item \rm{(a)} $\gamma_2(\Gamma \circ_v \Omega) = n\gamma_2(\Omega)$.
	\item \rm{(b)} $\gamma_2(\Omega - v) \geq \gamma_2(\Omega)$.
\end{Theorem}

\begin{Theorem} \label{18}
	For an isolate-free graph $\Gamma$, $\max\{ \gamma_{i}(\Gamma), \dim(\Gamma) \} \leq \gamma_{ri}(G) \leq \gamma_i(\Gamma)+ \dim(\Gamma)$.
\end{Theorem}

\begin{proof}
	It is obvious from the inequality that $\gamma(\Gamma) \leq \gamma_i(\Gamma)$ and Theorem \ref{1}.
\end{proof}

\begin{Theorem} \label{14}
	If $\gamma_t(\Gamma) \geq \dim(\Gamma)$, then $\gamma_t(\Gamma) \leq \gamma_{rt}(\Gamma) \leq \gamma_t(\Gamma)+ \dim(\Gamma)$.
\end{Theorem}

\begin{proof}
	Since the resolving total dominating set must resolve and total dominate the vertices of the graph it is obvious when $\gamma_t(\Gamma) \geq \dim(\Gamma)$, so $\gamma_t(\Gamma) \leq \gamma_{rt}(\Gamma) \leq \gamma_t(\Gamma)+ \dim(\Gamma)$. The upper bound could be exhibited by the graph which exists as $P_2 \circ_v K_3$. Consult Figure \ref{ex}.
\end{proof}

\begin{Theorem} \label{15}
	If $\gamma_t(\Gamma) \leq \dim(\Gamma)$, then $\dim(\Gamma) \leq \gamma_{rt}(\Gamma) \leq \gamma_t(\Gamma)+ \dim(\Gamma)$.
\end{Theorem}

\begin{proof}
	Since the resolving total dominating set must resolve and total dominate the vertices of the graph it is obvious when $\gamma_t(\Gamma) \leq \dim(\Gamma)$, then $\dim(\Gamma) \leq \gamma_{rt}(\Gamma) \leq \gamma_t(\Gamma)+ \dim(\Gamma)$. The upper bound could be exhibited by the graph which exists as $P_2 \circ_v K_3$. Consult Figure \ref{ex}.
\end{proof}

\begin{Theorem} \label{9}
	For every graph $\Gamma$, $\max\{ \gamma_t(\Gamma), \dim(\Gamma) \} \leq \gamma_{rt}(\Gamma) \leq \gamma_t(\Gamma)+ \dim(\Gamma)$.
\end{Theorem}

\begin{proof}
	From Theorem \ref{14} and Theorem \ref{15}, we get $\max\{ \gamma_t(\Gamma), \dim(\Gamma) \} \leq \gamma_{rt}(\Gamma) \leq \gamma_t(\Gamma)+ \dim(\Gamma)$. 
\end{proof}

\begin{Theorem} \label{16}
	If $\gamma_c(\Gamma) \geq \dim(\Gamma)$, then $\gamma_c(\Gamma) \leq \gamma_{rc}(\Gamma) \leq \gamma_c(\Gamma)+ \dim(\Gamma)$.
\end{Theorem}

\begin{proof}
	Since the resolving connected dominating set must resolve and dominate the vertices of the graph with the condition that the induced set must be connected, it is obvious when $\gamma_c(\Gamma) \geq \dim(\Gamma)$, so $\gamma_c(\Gamma) \leq \gamma_{rc}(\Gamma) \leq \gamma_c(\Gamma)+ \dim(\Gamma)$. The upper bound could be exhibited by the graph which exists as $P_2 \circ_v K_3$. Consult Figure \ref{ex}.
\end{proof}

\begin{Theorem} \label{17}
	If $\gamma_c(\Gamma) \leq \dim(\Gamma)$, then $\dim(\Gamma) \leq \gamma_{rc}(\Gamma) \leq \gamma_c(\Gamma)+ \dim(\Gamma)$.
\end{Theorem}

\begin{proof}
	Since the resolving connected dominating set must resolve and dominate the vertices of the graph with the condition that the induced set must be connected, it is obvious when $\gamma_c(\Gamma) \leq \dim(\Gamma)$, then $\dim(\Gamma) \leq \gamma_{rc}(\Gamma) \leq \gamma_c(\Gamma)+ \dim(\Gamma)$. The upper bound could be exhibited by the graph which exists as $P_2 \circ_v K_3$. Consult Figure \ref{ex}.
\end{proof}

\begin{Theorem} \label{10}
	For every graph $\Gamma$, $\max\{ \gamma_c(\Gamma), \dim(\Gamma) \} \leq \gamma_{rc}(\Gamma) \leq \gamma_c(\Gamma)+ \dim(\Gamma)$.
\end{Theorem}

\begin{proof}
	From Theorem \ref{16} and Theorem \ref{17}, we get $\max\{ \gamma_c(\Gamma), \dim(\Gamma) \} \leq \gamma_{rc}(\Gamma) \leq \gamma_c(\Gamma)+ \dim(\Gamma)$. 
\end{proof}

In Figure \ref{ex}, we could see that the resolving total domination and resolving connected domination numbers attain upper bound whereas in case of resolving domination it attains the lower bound. This led to an interesting theorem where we try to relate the resolving domination parameters for any graph $\Gamma$.

\begin{figure}[H]
	\centering
	\includegraphics[scale=0.6]{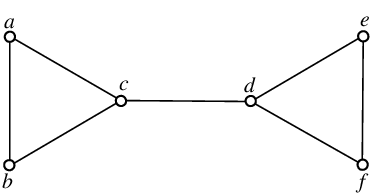}
	\quad \quad  \quad
	\includegraphics[scale=0.6]{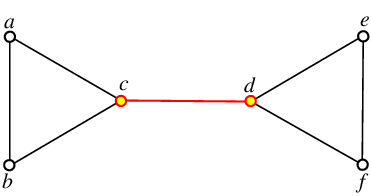}
	\quad \quad \quad
	\includegraphics[scale=0.6]{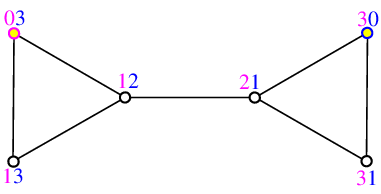}
	\caption{\textbf{(a)} $\Gamma$; \textbf{(b)} \textbf{TD} and \textbf{CD} set of $\Gamma$; \textbf{(c)} Resolving set of $\Gamma$}
	\label{ex}
\end{figure}

\begin{Theorem} 
	For any graph $\Gamma$, $\gamma_r(\Gamma) \leq \gamma_{rt}(\Gamma) \leq \gamma_{rc}(\Gamma)$.
\end{Theorem}

\begin{proof}
	Since $\gamma(\Gamma) \leq \gamma_t(\Gamma) \leq \gamma_c(\Gamma)$ it is also true by the definition for $\gamma_r(\Gamma) \leq \gamma_{rt}(\Gamma) \leq \gamma_{rc}(\Gamma)$.
\end{proof}

\begin{Theorem} \label{11}
	$\gamma(FCN(l)) = 4(\gamma(FCN(l-1)))-2$ for $l \geq 1$.
\end{Theorem}

\begin{proof}
	\begin{sloppypar}
		By Theorem \ref{2}, we get $\gamma(FCN(l)) \in \{4(\gamma(FCN(l-1)))-2, 4(\gamma(FCN(l-1)))\}$ for $l \geq 1$. Since $4(\gamma(FCN(l-1)))-2$ is minimum, we get $\gamma(FCN(l)) = 4(\gamma(FCN(l-1)))-2$ for $l \geq 1$. The \textbf{DS} is $D_{FCN(l)} =\{ab \ || \ D_{FCN(l-1)} : a,b \in \{0,1\} \} \setminus \{1110(01)^{(l-1)}, 0010(01)^{(l-1)}\}$ for $l \ge 1$, where $D_{FCN(1)}=\{1101, 1010, 1001, 0110, 0101, 0001\}$. See Figure \ref{FD1}.
	\end{sloppypar}
\end{proof}

\begin{figure}[H] 
	\centering
	\includegraphics[scale=1]{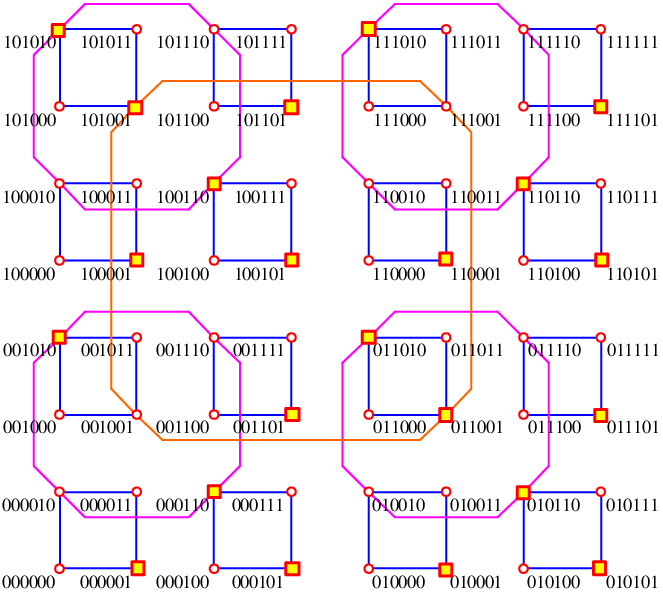}  
	\caption{\textbf{DS} and \textbf{IDS} of $FCN(2)$} \label{FD1}
\end{figure}

\begin{Corollary} \label{19}
	$\gamma_i(FCN(l)) = 4(\gamma_i(FCN(l-1)))-2$ for $l \geq 1$.
\end{Corollary}

\begin{proof}
	\begin{sloppypar}
		By Theorem \ref{20}, we get $\gamma_i(FCN(l)) \in \{4(\gamma_i(FCN(l-1)))-2, 4(\gamma_i(FCN(l-1)))\}$ for $l \geq 1$. Since $4(\gamma_i(FCN(l-1)))-2$ is minimum, we get $\gamma_i(FCN(l)) = 4(\gamma_i(FCN(l-1)))-2$ for $l \geq 1$. The \textbf{IDS} is $D_{FCN(l)} =\{ab \ || \ D_{FCN(l-1)} : a,b \in \{0,1\} \} \setminus \{1110(01)^{(l-1)}, 0010(01)^{(l-1)}\}$ for $l \ge 1$, where $D_{FCN(1)}=\{1101, 1010, 1001, 0110, 0101, 0001\}$. See Figure \ref{FD1}.
	\end{sloppypar}
\end{proof}

\begin{Theorem} \label{21}
	$\gamma_t(FCN(1)) = 8$.
\end{Theorem}

\begin{proof}
	\begin{sloppypar}
		By Theorem \ref{3}, we get $\gamma_t(FCN(1)) \in \{4, 6, 6, 8\}$. Let $I(FCN(1))$ be the set of all possible sub graphs of $FCN(1)$ with at least one isolated vertex. If $\gamma_t(FCN(1)) \in \{4, 6\}$, then $\exists$ $v \in V(FCN(1))$ but $v \notin N[D_{FCN(1)}]$ or $FCN(1)[D_{FCN(1)}] \subseteq I(FCN(1))$. This is a contradiction to the definition of the total dominating set. Thus $\gamma_t(FCN(1))=8$. The \textbf{TDS} is $D_{FCN(1)}=\{1111, 1110, 1010, 1011, 0111,  0110, 0011, 0010\}$. See Figure \ref{FD7}.
	\end{sloppypar}
\end{proof}

\begin{figure}[H] 
	\centering
	\includegraphics[scale=0.8]{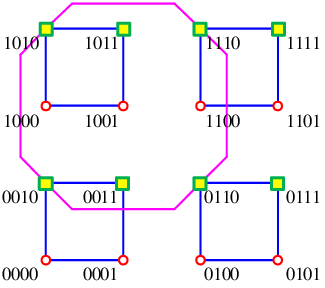}  
	\caption{\textbf{TDS} of $FCN(1)$} \label{FD7}
\end{figure}

\begin{Theorem} \label{12}
	$\gamma_t(FCN(l)) = 4(\gamma_t(FCN(l-1)))-2$ for $l \geq 2$. 
\end{Theorem}

\begin{figure}[H] 
	\centering
	\includegraphics[scale=1]{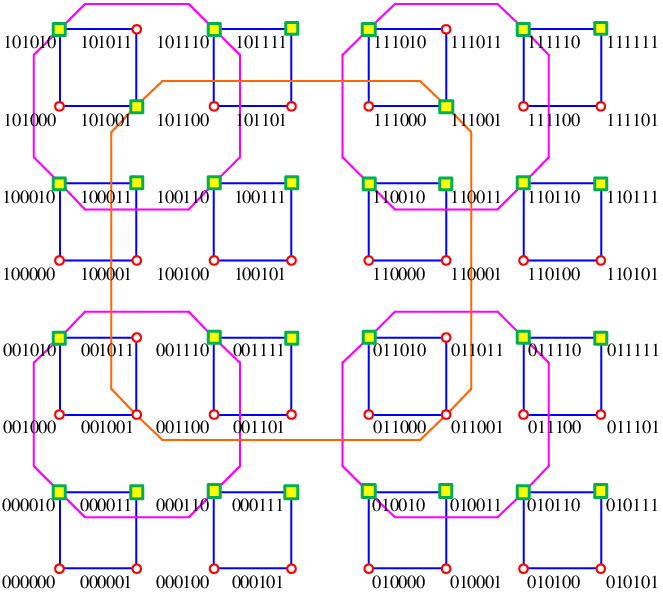}  
	\caption{\textbf{TDS} of $FCN(2)$} \label{FD2}
\end{figure}

\begin{proof}
	\begin{sloppypar}
		By Theorem \ref{3}, we get $\gamma_t(FCN(l)) \in \{4(\gamma_t(FCN(l-1))-1), 4(\gamma_t(FCN(l-1)))-2, 4(\gamma_t(FCN(l-1)))-2, 4(\gamma_t(FCN(l-1)))\}$. Let $I(FCN(l))$ be the set of all possible sub graphs of $FCN(l)$ with at least one isolated vertex. Since they are three minimum values of the \textbf{TDS} we have, $\gamma_t(FCN(l)) \in \{4(\gamma_t(FCN(l-1))-1), 4(\gamma_t(FCN(l-1)))-2, 4(\gamma_t(FCN(l-1)))-2\}$. If $\gamma_t(FCN(l))=4(\gamma_t(FCN(l-1))-1)$ then $\exists$ $v \in V(FCN(l))$ but $v \notin N[D_{FCN(l)}]$ or $FCN(l)[D_{FCN(l)}] \subseteq I(FCN(l))$. This is a contradiction to the definition of the total dominating set. Thus $\gamma_t(FCN(l)) = 4(\gamma_t(FCN(l-1)))-2$ for $l \geq 2$. The \textbf{TDS} is $D_{FCN(l)} = \{ab \ || \   D_{FCN(l-1)} : a,b \in \{0,1\} \} \cup \{ 1010(01)^{l-1}, 1110(01)^{l-1}\} \setminus \{ (1010)(01)^{l-2}(11), (1110)(01)^{l-2}(11), (0110)(01)^{l-2}(11), (0010)(01)^{l-2}(11)\}$ for $l \ge 2$, where $D_{FCN(1)}=\{1111, 1110, 1010, 1011, 0111,  0110, 0011, 0010\}$. See Figure \ref{FD2}.
	\end{sloppypar}
\end{proof}

\begin{Theorem} \label{13}
	$\gamma_c(FCN(l)) = 4(\gamma_c(FCN(l-1))+1)$ for $l \geq 1$.
\end{Theorem}

\begin{proof}
	\begin{sloppypar}
		By Theorem \ref{4}, $\gamma_c(FCN(l)) \in \{4(\gamma_c(FCN(l-1))), 4(\gamma_c(FCN(l-1))+1) \}$. If $\gamma_c(FCN(l)) = 4(\gamma_c(FCN(l-1)))$ then $FCN(l)[D_{FCN(l)}]$ will be an induced disconnected graph or $\exists$ $v \in V(FCN(l))$ but not in $N[D_{FCN(l)}]$. A contradiction to the definition of the connected dominating set. Thus $\gamma_c(FCN(l)) = 4(\gamma_c(FCN(l-1))+1)$ for $l \geq 1$. The \textbf{CDS} is $D_{FCN(l)} = \{ab \ || \ D_{FCN(l-1)} : a,b \in \{0,1\}\} \cup \{ 0010(01)^{l-1}, 0110(01)^{l-1}, 1010(01)^{l-1}, 1110(01)^{l-1} \}$ for $l \ge 1$, where $D_{FCN(1)}=\{1111, 1110, 1011, 1010, 0111, 0110, 0011, 0010\}$. See Figure \ref{FD3}.
	\end{sloppypar}
\end{proof}

\begin{Theorem}
	$\gamma_{\times 2}(FCN(1)) = 12$.
\end{Theorem}

\begin{proof}
	\begin{sloppypar}
		It is obvious that for $FCN(0)$, the double domination number is 3 since it is isomorphic to $C_4$. By Theorem \ref{6}, we get $\gamma_{\times 2}(FCN(1)) = 4(\gamma_{\times 2}(FCN(0))) = 12$. The \textbf{DDS} is $D_{FCN(1)}=\{1111, 1110, 1101, 1011, 1010, 1001, 0111, 0110, 0101, 0011, 0010, 0001 \}$. See Figure \ref{FDD3}.
	\end{sloppypar}
\end{proof}

\begin{Theorem}
	$\gamma_{\times 2}(FCN(l)) = 4(\gamma_{\times 2}(FCN(l-1))-1)$ for $l \geq 2$.
\end{Theorem}

\begin{proof}
	\begin{sloppypar}
		By Theorem \ref{5}, we get $\gamma_{\times 2}(FCN(l)) \in \{4\gamma_{\times 2}(FCN(l-1)), 4(\gamma_{\times 2}(FCN(l-1))), 4(\gamma_{\times 2}(FCN(l-1)))-2, 4(\gamma_{\times 2}(FCN(l-1)))-2, 4(\gamma_{\times 2}(FCN(l-1))-1), 4(\gamma_{\times 2}(FCN(l-1)))-5\}$. Here, there are two minimum values for the \textbf{DDS}. That is $\gamma_{\times 2}(FCN(l)) \in \{4(\gamma_{\times 2}(FCN(l-1))-1), 4(\gamma_{\times 2}(FCN(l-1)))-5\}$. If $\gamma_{\times 2}(FCN(l)) = 4(\gamma_{\times 2}(FCN(l-1)))-5$ then $\exists$ $v \in V(FCN(l))$ where, $|N[v]\cap D_{FCN(l)}| \leq 1$. Thus $\gamma_{\times 2}(FCN(l)) = 4(\gamma_{\times 2}(FCN(l-1))-1)$ for $l \geq 2$. The \textbf{DDS} is $D_{FCN(l)} = \{ ab \ || \ D_{FCN(l-1)} : a,b \in \{0,1\} \} \setminus \{(1110)(01)^{l-2}(11), (1010)(01)^{l-2}(11), (0110)(01)^{l-2}(11), (0010)(01)^{l-2}(11)\} $ for $l \ge 2$, where $D_{FCN(1)}= \{1111, 1110, 1101, 1011, 1010, 1001, 0111, 0110, 0101, 0011, 0010, 0001 \}$. See Figure \ref{FD4}.
	\end{sloppypar}
\end{proof}

\begin{figure}[H] 
	\centering
	\includegraphics[scale=1]{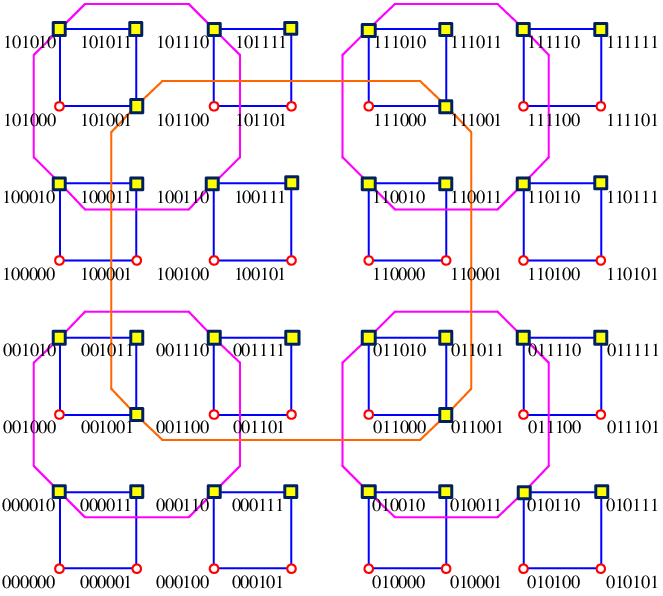}  
	\caption{\textbf{CDS} of $FCN(2)$} \label{FD3}
\end{figure}

\begin{figure}[H] 
	\centering
	\includegraphics[scale=1]{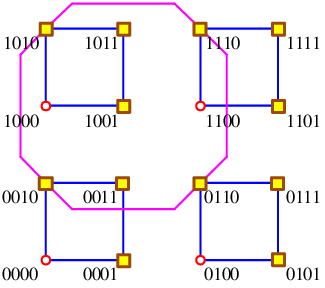}  
	\caption{\textbf{DDS} of $FCN(1)$} \label{FDD3}
\end{figure}

\begin{Theorem}
	$\gamma_{2}(FCN(l)) = 4(\gamma_2(FCN(l-1)))$ for $l \geq 1$.
\end{Theorem}

\begin{proof}
	\begin{sloppypar}
		From Theorem \ref{7} and Theorem \ref{8}, we get $\gamma_{2}(FCN(l)) = 4(\gamma_2(FCN(l-1)))$ for $l \geq 1$. The \textbf{2DS} is $D_{FCN(l)}=\{N(s_1 \ldots s_{2l}01): s_1, s_2, \ldots, s_{2l} \in \{0,1\}\}$ for $l \ge 1$. See Figure \ref{FD5}.
	\end{sloppypar}
\end{proof}

\begin{figure}[H] 
	\centering
	\includegraphics[scale=1]{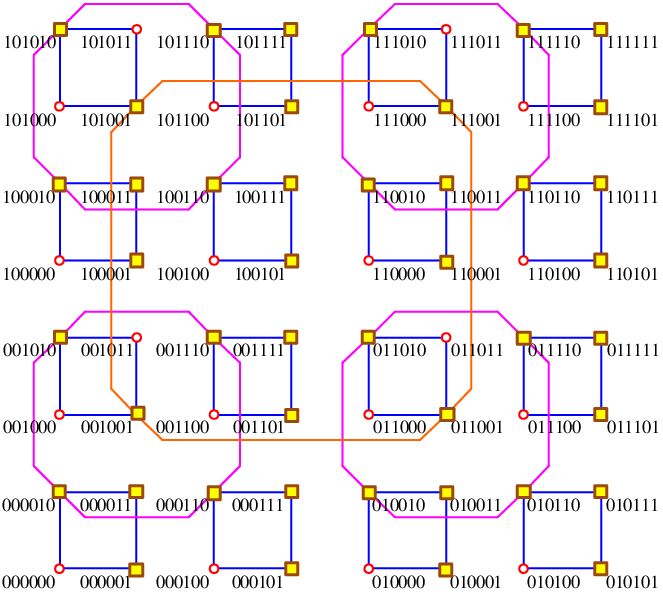}  
	\caption{\textbf{DDS} of $FCN(2)$} \label{FD4}
\end{figure}

\begin{Theorem} \label{23}
	$\gamma_r(FCN(l)) = 4(\gamma_2(FCN(l-1)))$ for $l \geq 1$.
\end{Theorem}

\begin{proof}
	\begin{sloppypar}
		From Theorem \ref{mmdfcn}, Theorem \ref{1} and Theorem \ref{11}, we have $\max\{4^l, 4(\gamma(FCN(l-1)))-2\} \leq \gamma_r(FCN(l)) \leq 4^l+4(\gamma(FCN(l-1)))-2$. Suppose $\gamma_r(FCN(l))=4^l$, then $\exists$ $v \in V(FCN(l))$ which is a two degree twin vertex, such that $v \notin N[D_{FCN(l)}]$. This is a contradiction to the definition of resolving domination. Suppose $\gamma_r(FCN(l)) = 4(\gamma(FCN(l-1)))-2$, by Theorem \ref{22} and \ref{11}, the vertices of the set does not contain twin vertices which are necessary for resolving the graph. This is a contradiction to the definition of the resolving domination. Thus the set must contain all twin vertices except one from each twin set and must dominate the vertices of $FCN(l)$. Thus, in every disjoint $C_4$ of $FCN(l)$, we need at least two vertices and one must be from the twin set. This implies that $\gamma_r(FCN(l)) = 4(\gamma_2(FCN(l-1)))$ for $l \geq 1$. The \textbf{RDS} is $D_{FCN(l)}=\{N(s_1 \ldots s_{2l}01): s_1, s_2, \ldots, s_{2l} \in \{0,1\}\}$ for $l \ge 1$. See Figure \ref{FD5}.
	\end{sloppypar}
\end{proof}

\begin{Corollary}
	$\gamma_{ri}(FCN(l)) = 4(\gamma_2(FCN(l-1)))$ for $l \geq 1$.
\end{Corollary}

\begin{proof}
	\begin{sloppypar}
		From Theorem \ref{mmdfcn}, Theorem \ref{18} and Corollary \ref{19}, we get $\max\{4^l, 4(\gamma_{i}(FCN(l-1)))-2\} \leq \gamma_{ri}(FCN(l)) \leq 4^l+4(\gamma_{i}(FCN(l-1)))-2$ for $l \geq 1$. The same argument constructed in Theorem \ref{23} holds here. Thus $\gamma_{ri}(FCN(l))=4(\gamma_2(FCN(l)))$ for $l \geq 1$ and the set is independent and resolving set. The \textbf{RIDS} is $D_{FCN(l)}=\{N(s_1 \ldots s_{2l}01): s_1, s_2, \ldots, s_{2l} \in \{0,1\}\}$ for $l \ge 1$. See Figure \ref{FD5}.
	\end{sloppypar}
\end{proof}

\begin{Theorem}
	$\gamma_{rt}(FCN(1)) = 8$.
\end{Theorem}

\begin{proof}
	\begin{sloppypar}
		From Theorem \ref{mmdfcn}, Theorem \ref{9} and Theorem \ref{21} we get $\gamma_{rt}(FCN(1)) = 8$, since it must be both resolving and total dominating set. The \textbf{RTDS} is $D_{FCN(1)}=\{1111, 1110, 1011, 1010, 0111, 0110, 0011, 0010\}$. See Figure \ref{FD7}.
	\end{sloppypar}
\end{proof}

\begin{figure}[H] 
	\centering
	\includegraphics[scale=1]{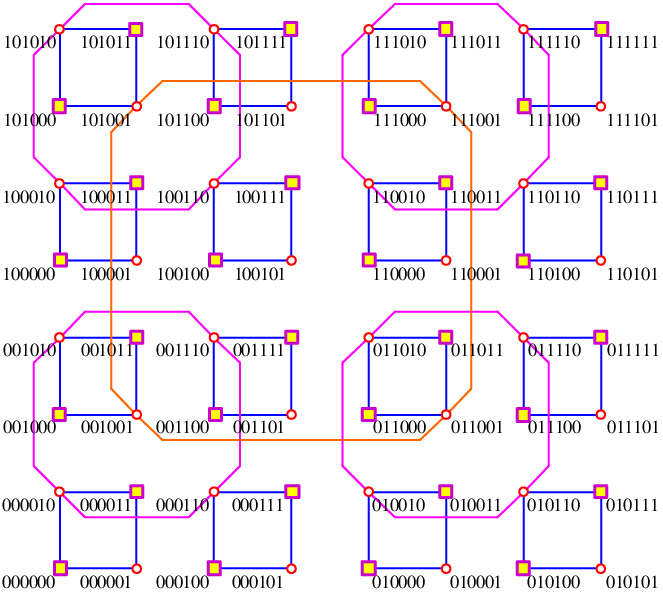}  
	\caption{\textbf{2DS}, \textbf{RDS} and \textbf{RIDS} of $FCN(2)$} \label{FD5}
\end{figure}

\begin{Theorem}
	$\gamma_{rt}(FCN(l)) = 4(\gamma_{rt}(FCN(l-1)))$ for $l \geq 2$.
\end{Theorem}

\begin{proof}
	\begin{sloppypar}
		From Theorem \ref{mmdfcn}, Theorem \ref{9} and Theorem \ref{12}, we have $\max\{4^l, 4(\gamma_t(FCN(l-1)))-2\} \leq \gamma_{rt}(FCN(l)) \leq 4^l + 4(\gamma_t(FCN(l-1)))-2$. The same argument constructed in Theorem \ref{23} holds here. But we need to choose the twins except one from each twin set including its one neighbor. Thus $\gamma_{rt}(FCN(l))=4(\gamma_{rt}(FCN(l-1)))$ for $l \geq 2$ and the set is total dominating and resolving set. The \textbf{RTDS} is $D_{FCN(l)} = \{ab \ || \   D_{FCN(l-1)} : a,b \in \{0,1\} \}$ for $l \ge 2$, where $D_{FCN(1)}=\{1111, 1110, 1011, 1010, 0111, 0110, 0011, 0010\}$. See Figure \ref{FD6}.
	\end{sloppypar}
\end{proof}

\begin{Theorem}
	$\gamma_{rc}(FCN(l)) = 4(\gamma_c(FCN(l-1))+1)$ for $l \geq 1$.
\end{Theorem}

\begin{proof}
	\begin{sloppypar}
		From Theorem \ref{mmdfcn}, Theorem \ref{10} and Theorem \ref{13}, we get $\gamma_{rc}(FCN(l)) = 4(\gamma_c(FCN(l-1))+1)$ for $l \geq 1$. The \textbf{RCDS} is $D_{FCN(l)} = \{ab \ || \ D_{FCN(l-1)} : a,b \in {0,1}\} \cup \{ 0010(01)^{l-1}, 0110(01)^{l-1}, 1010(01)^{l-1}, 1110(01)^{l-1} \}$ for $l \ge 1$ where $D_{FCN(1)}=\{1111, 1110, 1011, 1010, 0111, 0110, 0011, 0010\}$. See Figure \ref{FD3}.
	\end{sloppypar}
\end{proof}

\begin{figure}[H] 
	\centering
	\includegraphics[scale=1]{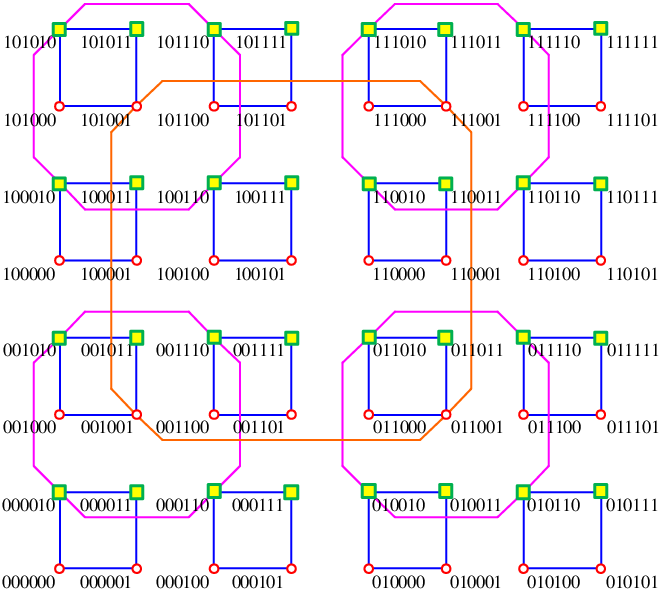}  
	\caption{\textbf{RTDS} of $FCN(2)$} \label{FD6}
\end{figure}

\section{Conclusion}

Studying domination in the context of rooted product graphs has provided new insights into domination parameters, such as the domination number and the change when two graphs are combined through a rooted product operation. This work has contributed to understand the original properties of the graph influencing the domination properties of resultant graphs. Specifically, for rooted product graphs, the domination behavior has been linked to the properties of the root vertex and interaction with other vertices. These insights can help design more efficient algorithms and solve graph-theoretical problems in network theory, biology, and computer science. Future research includes developing algorithms, root selection strategies, and parameter generalization. By addressing these future research directions, a more comprehensive understanding of domination in rooted product graphs can be developed, leading to theoretical advancements and practical applications across various fields. \\
\textbf{Author Contributions}: {Conceptualization, S.P. and A.K.A; methodology, S.P.; software, A.K.A. and M.A.; validation, S.P., A.K.A. and M.A.; formal analysis, S.P. and A.K.A; investigation, S.P. and A.K.A; writing---original draft preparation, A.K.A.; writing---review and editing, S.P.; visualization, S.P., A.K.A, and M.A.; supervision, S.P. All authors have read agreed to the published version of the manuscript.}\\
\textbf{Funding}: {This research received no external funding.}\\
\textbf{Data Availability}: {No new data were created or analyzed in this study. Data sharing is not applicable to this article.}\\
\textbf{Conflicts of Interest}: {The authors declare no conflicts of interest.}

\end{document}